\documentclass[11pt]{amsart}
\usepackage{pgf,tikz,pgfplots,color}
\pgfplotsset{compat=1.15}
\usepackage{mathrsfs}
\usetikzlibrary{arrows}
\usepackage{fancyhdr}
\usepackage{lastpage}

\usepackage[margin=0.9in]{geometry}

\usepackage{hyperref}
\hypersetup{
    colorlinks=true,
    linkcolor=blue,
    citecolor=brown,
    filecolor=magenta,
    urlcolor=blue,
}
\usepackage{amsmath}
\usepackage{amsfonts}
\usepackage{amssymb}
\usepackage{amsthm}
\usepackage{setspace}
\usepackage{inputenc}
\usepackage{comment}
\usepackage{mathtools}
\usepackage[shortlabels]{enumitem}
\usepackage{tikz-cd}

\newtheorem{thm}{Theorem}

\newtheorem{lem}[thm]{Lemma}
\newtheorem{prop}[thm]{Proposition}

\theoremstyle{definition}

\newtheorem*{defn*}{Definition}

\newtheorem{rmk}[thm]{Remark}

\newtheorem*{rmk*}{Remark}
\newtheorem*{ack}{Acknowledgment}

\allowdisplaybreaks

\newcommand{\1}{\mathbf{1}}

\newcommand{\Bs}{\mathscr{B}}
\newcommand{\C}{\mathbb{C}}
\newcommand{\Fs}{\mathscr{F}}

\newcommand{\R}{\mathbb{R}}

\newcommand{\Ss}{\mathscr{S}}

\newcommand{\Z}{\mathbb{Z}}

\newcommand{\eps}{\varepsilon}

\newcommand{\Span}{\operatorname{span}}

\newcommand{\loc}{\mathrm{loc}}

\title[Non-compact Sobolev-Lorentz Embeddings]{A Bourgain-Gromov problem on non-compact Sobolev-Lorentz embeddings}

\author[Chian Yeong Chuah]{Chian Yeong Chuah}
\address{Department of Mathematics, University of Georgia, Athens, GA 30602, United States}
\email{chianyeong.chuah@uga.edu}

\author[Jan Lang]{Jan Lang}
\address{\parbox{\linewidth}{Department of Mathematics, The Ohio State University, Columbus, OH 43210, United States\\
Department of Mathematics, Faculty of Electrical Engineering, Czech Technical University in Prague, 166 27 Prague, Czech Republic\vspace{0.055in}}}
\email{lang@math.osu.edu}

\author[Liding Yao]{Liding Yao} 
\address{Department of Mathematical Sciences,
	Purdue University Fort Wayne, Fort Wayne, IN 46805, United States} 
\email{yao388@purdue.edu}


\begin{document}

\begin{abstract}
    We study the non-compact Sobolev embeddings into the optimal scale of Lorentz spaces,
    \[W_0^mL^{p,q}(\Omega) \to L^{\frac{dp}{d - mp},r}(\Omega),\]
    where $\Omega \subseteq \mathbb{R}^d$, $1 \leq m < d$, $0<q<r\le\infty$ with $1<p<\frac dm$ or $p=q=1$. We show that these embeddings are finitely strictly singular with certain upper bounds on the decay rate of the Bernstein numbers. We reduce the Sobolev embeddings to embeddings of Besov spaces and sequence spaces, which simplifies the previous methods by Bourgain-Gromov and Lang-Mihula.
\end{abstract}

\maketitle

\section{Introduction and Main result}

Let us start with the following non-compact Sobolev embeddings into the scale of Lorentz spaces:
\begin{equation*}
    W_0^{m,p} (\Omega) \to L^{\frac{dp}{d - mp},q}(\Omega),
\end{equation*}
where $\Omega \subseteq \mathbb{R}^d$, $1 \leq m \leq d$, $p \in [1, d/m)$, $q\in[p,\infty]$.

It is known that these Sobolev embeddings into the target Lorentz spaces $L^{p^*,q}$, with $p \leq q \leq \infty$, are {\bf maximally non-compact} (i.e., the operator norm of the embedding equals the ball measure of non-compactness). See \cite{LMOP} and the references therein.

Let us recall that the {\bf ball measure of non-compactness} $\alpha(T)$ for a linear map $T: X \to Y$ acting between Banach spaces $X$ and $Y$ is defined as the infimum of radii $\rho > 0$ such that there exists a finite set of balls in $Y$ with radii $\rho$ covering $T(B_X)$, where $B_X$ denotes the open unit ball in $X$ centered at the origin.

It is also well-known that $L^{p^*,q} \subset L^{p^*,r}$ whenever $q < r$ and that $L^{p^*,p}$ is the smallest target space for the Sobolev space $W^{m,p}=W^mL^{p}$, not only within the scale of Lorentz spaces but also across all rearrangement-invariant spaces (e.g., Orlicz spaces, Lebesgue spaces, etc.).

A fundamental question that arises in this context is how to \textbf{quantify} the \textbf{quality} of non-compactness of the Sobolev embedding when mapping into different optimal or almost optimal Lorentz spaces.
We will focus on studying non-compactness of Sobolev embeddings involving Lorentz spaces:
\begin{equation*}
    W_0^mL^{p,q}(\Omega) \to L^{\frac{dp}{d - mp},r}(\Omega), \mbox{ for } q\le r.
\end{equation*}

Clearly, relying solely on the ball measure of non-compactness is insufficient to capture finer qualitative differences. This necessitates the use of additional quantities. Furthermore, classical quantities such as entropy numbers, approximation numbers, and Kolmogorov numbers fail to provide insight in this setting, since their asymptotic behavior coincides with the ball measure of non-compactness.

A more refined notion for describing the {\bf quality} of non-compactness, beyond the measure of non-compactness, appears to be related to the concepts of {\bf strictly singular operators}, {\bf finitely strictly singular operators}, and {\bf Bernstein numbers}. For their definitions, see Section 2 and also \cite{LefevrePiazzaFSSApplication} and \cite{LangMusil}.

In their seminal work, Bourgain and Gromov (1987) \cite{BouGro} showed that $W^{1,1} ((0,1)^d) \to L^{\frac d{d-1}}((0,1)^d)$ is finitely strictly singular and provided upper bounds for the decay rates of Bernstein numbers. Later, Lang and Mihula \cite{LangMihula} demonstrated that the {\bf optimal} Sobolev embedding $W^{k,p} (\Omega) \to L^{p^*,p}(\Omega)$ is {\bf not} strictly singular (i.e., there exists an infinite-dimensional subspace on which the embedding is invertible). They also showed that the embedding $W^{k,p} (\Omega) \to L^{p^*}(\Omega)$ is finitely strictly singular and established sharp estimates for the decay of Bernstein numbers.

However, extending the methods used in these works to a more general setting is challenging or impossible, particularly when $\Omega $ is unbounded or when derivatives belong to a Lorentz space. Moreover, these results leave open questions about the behavior of embeddings when the target spaces differ from $L^{p^*}$ and $L^{p^*,p}$.

In this paper, we address these open questions for the embedding $W_0^m L^{p,q} \hookrightarrow L^{p^*,r}$. Additionally, the boundedness assumption on $\Omega$ in \cite{BouGro} and \cite{LangMihula} can be significantly relaxed: our result holds on arbitrary open subsets of $\R^d$.

\begin{thm}\label{Thm::MainThm}
    Let  $m\in\Z_+$, $1\le p<\frac dm$,  $p^*=\frac{dp}{d-mp},$ and $\Omega\subseteq\R^d$ be an open subset (not necessarily bounded). Suppose $0<q<r\le\infty$.
    In the case $p=1$ we assume additionally $q=1$.

    Then, the embedding $W_0^m L^{p,q}(\Omega)\hookrightarrow L^{p^*,r}(\Omega)$ has the following decay of Bernstein numbers: for every $\eps>0,$ there exists $C_{\varepsilon}>0$ such that:
    \begin{equation}\label{Eqn::MainBnEst}
        b_n\big(W_0^m L^{p,q}(\Omega)\hookrightarrow L^{p^*,r}(\Omega)\big)\le C _\eps n^{-\min(\frac1{p+\eps},\frac1q)(1-\max(\frac qr,\frac p{p^*}+\eps))}.
    \end{equation}

    In particular, the embedding is finitely strictly singular.
\end{thm}
Here when $p=1$ the assumption $q\le 1$ is required so that $W^mL^{p,q}(\Omega)\subset L^1_\loc(\Omega)$. See also Remark~\ref{Rmk::Spector}.

When $p<q<r<p^*$ we obtain the sharp asymptotic estimate, see Proposition~\ref{Prop::SobSharp}:
\begin{equation*}
    C^{-1} n^{-(\frac1q-\frac1r)}\le b_n\big(W_0^m L^{p,q}(\Omega)\hookrightarrow L^{p^*,r}(\Omega)\big)\le C n^{-(\frac1q-\frac1r)}.
\end{equation*}

Although our method, generally,  does not yield an optimal decay of the Bernstein numbers, the finite strict singularity result is new, even for the case $W_0^kL^p(\Omega)\hookrightarrow L^{p^*}(\Omega)$ when $\Omega$ is an unbounded domain. Also, our approach yields a simpler proof than that given in \cite{LangMihula} and \cite{BouGro}.

It is worth noting that on bounded domains, it may be possible to obtain certain one-sided estimates for Bernstein numbers from \cite{LangMihula} using interpolation techniques, as pointed out to us by Z. Mihula.

\section{Definitions and Proofs}

Let us recall some definitions and notations before proving the main theorem.

Given a bounded linear map $T:X\to Y$ between two quasi-Banach spaces, the \textit{$n$-th Bernstein number} $b_n(T)$ is defined as
\begin{equation*}
    b_n(T):=\sup\big\{\inf_{x\in X_n:\|x\|_X=1} \|Tx\|_{Y}:X_n \subseteq X\text{ is an $n$-dimensional subspace}\}.
\end{equation*}
We say that $T$ is \textit{finitely strictly singular}, if $\lim_{n\to\infty}b_n(T)=0$.

Note that the Bernstein numbers coincide with the smallest injective strict $s$-numbers; see \cite[Theorem~4.6]{P:74}.

We say the operator $T$ is  \emph{strictly singular} if there is no infinite-dimensional subspace $Z \subseteq X$ such that the restriction $T\rvert_{Z}$ is an isomorphism from $Z$ onto $T(Z)$. Equivalently,
\begin{equation*}
    \inf_{x\in Z:\|x\|_X=1} \|Tx\|_{Y}=0,\quad\text{for every infinite dimensional subspace }Z\subseteq X.
\end{equation*}

The relationships among these notions and that of compactness of $T$  are summarized by the following diagram:
\begin{equation*}
    T\text{ is compact}
    \;\Longrightarrow\;
    T\text{ is finitely strictly singular}
    \;\Longrightarrow\;
    T\text{ is strictly singular},
\end{equation*}
and each reverse implication fails in general. For further details and additional background, we refer the reader to \cite{PlichkoStrictlySingular} and \cite{LefevrePiazzaFSSApplication}.

Given a measure space $(\Omega,\mu)$, for $0<p<\infty$ and $0<r\le\infty$ we define the Lorentz space $L^{p,r}(\mu)$ by the set of all measurable functions $f:\Omega\to\R$ such that
\begin{equation*}
    \|f\|_{L^{p,r}(\mu)}=p^{\frac1r}\Big(\int_0^\infty t^{r-1}\mu\{x:|f(x)|>t\}^\frac rpdt\Big)^{\frac{1}{r}}=\Big(\int_0^\infty t^{\frac{r}{p}-1}f^*(t)^{r} dt\Big)^{\frac{1}{r}}<\infty.
\end{equation*}
Here $f^*(t):=\inf\{\alpha>0:\mu\{x:|f(x)|>\alpha\}\leq t\}$ is the decreasing rearrangement of $f$. When $r=\infty$ we replace the integral by the supremum.

One can check that $L^{p,p}(\mu)=L^p(\mu)$ with the same (quasi-)norm and $L^{p,r_0}(\mu)\subset L^{p,r_1}(\mu)$ for $0<p\le\infty$ and $0<r_0<r_1\le\infty$.

Let $\Omega\subseteq\R^d$ be an open subset. When $1<p\le\infty$ or $p=1\ge r$, we define the Sobolev-Lorentz space $ W^{m}L^{p,r}(\Omega)$ for $m\ge0$ by the set of $L^1_\loc(\Omega)$ functions such that
\begin{equation}\label{Eqn::SobLorNorm}
    \|f\|_{W^{m}L^{p,r}(\Omega)}:=\sum_{|\alpha|\le m}\|\partial^\alpha f\|_{L^{p,r}(\Omega)}<\infty.
\end{equation}
We denote by $W_0^mL^{p,r}(\Omega)$ the closure of $C_c^\infty(\Omega)$ under $\|\cdot\|_{W^mL^{p,r}(\Omega)}$. Notice that $W_0^mL^{p,r}(\Omega)$ can be seen as a subspace of both $W^mL^{p,r}(\Omega)$ and  $W_0^mL^{p,r}(\R^d)$.

Note that in some literature, these spaces are labeled as $W^{m,(p,r)}(\Omega)$.
In the case $p=1$, we require $r\le 1$ in order to have  $L^{1,r}\subset L^1 $ so that the elements in Sobolev-Lorentz spaces make sense as distributions.

The key essence of our proof is to decompose the Sobolev-Lorentz embeddings into Besov spaces.

Recall that for $s\in\R$ and $0<p,q\le\infty$, the \textit{Besov space} $\Bs_{pq}^s(\R^n)$ consists of all (tempered distributions) $f$ such that
\begin{align*}
    \|f\|_{\Bs_{pq}^s(\R^n)} = \|f\|_{\Bs_{pq}^s(\R^n;\varphi)}& := \Big( \sum_{j=0}^{\infty} \| 2^{js}\varphi_j\ast f\|_{L^p(\R^n)}^q \Big)^{1/q}<\infty.
\end{align*}
    Here $(\varphi_j)_{j=0}^\infty\subset\Ss(\R^n)$ is a fixed family of Schwartz functions whose Fourier transforms $\hat\varphi_j(\xi)=\int \varphi_j(x)e^{-2\pi ix\xi}dx$ satisfy $\1_{B(0,1)}\le\hat\varphi_0\le\1_{B(0,2)}$ and $\hat\varphi_j(\xi)=\hat\varphi_0(2^{-j}\xi)-\hat\varphi_0(2^{1-j}\xi)$ for $j\ge1$. When $q=\infty$, we use the usual supremum modification in the norm.

Different choices of $(\varphi_j)_j$ result in different but equivalent norms (see e.g. \cite[Proposition~2.3.2]{TriebelTheoryOfFunctionSpacesI}). Thus, the space is independent of $(\varphi_j)_j$ even though its norm is dependent on the choice of functions.

It is important that we have equivalent norms via derivatives, see e.g. \cite[Theorem~2.3.8(i)]{TriebelTheoryOfFunctionSpacesI}: for $p,q\in(0,\infty]$, $s,r\in\R$ and $m\ge1$,
\begin{equation}\label{Eqn::BsNormDer}
    \|(I-\Delta)^{r/2}f\|_{\Bs_{pq}^{s-r}}\approx_{p,q,s,r}\|f\|_{\Bs_{pq}^s}\approx_{p,q,s,m}\sum_{|\alpha|\le m}\|\partial^\alpha f\|_{\Bs_{pq}^{s-m}}.
\end{equation}

\begin{lem}\label{Lem::SobEmbed}
    Let $m\in\Z_+$, $1\le p< \infty$ and $0<q\le\infty$.
    \begin{enumerate}[(i)]
        \item\label{Item::SobEmbed::1} When $1<p<\infty$, $W^m L^{p,q}(\R^d)\hookrightarrow\Bs_{p^* q}^s(\R^d)$ for $s<m$ and $p^*\in(p,\infty]$ such that $\frac1{p^*}=\frac1p-\frac{m-s}d$.
        \item\label{Item::SobEmbed::2} When $p=1$, $W^{m,1}(\R^d) \hookrightarrow \Bs_{p^*1}^s(\R^d)$ for $m-d\le s<m$ and $\frac1{p^*}=1-\frac{m-s}d$.
        \item\label{Item::SobEmbed::3} When $1<p<\infty$, $\Bs_{\tilde pq}^s(\R^d)\hookrightarrow L^{p,q}(\R^d)$ for $s>0$ and $\frac1{\tilde p}=\frac1p+\frac sd$.
    \end{enumerate}
\end{lem}
\begin{proof}
    \ref{Item::SobEmbed::1} and \ref{Item::SobEmbed::3} can be done by real interpolation with classical Sobolev embeddings. Indeed for $0<p\le\infty$, $0<p_0,p_1<\infty$, $s_0,s_1\in\R$, $0<q,q_0,q_1\le\infty$, $0<\theta<1$ such that $p_0\neq p_1$, $s_0\neq s_1$, $\frac1{p_\theta}=\frac{1-\theta}{p_0}+\frac\theta{p_1}$ and $s_\theta=(1-\theta)s_0+\theta s_1$,
    \begin{equation}
        \label{Eqn::SobEmbed::RealInterpo}(L^{p_0}(\R^d),L^{p_1}(\R^d))_{\theta,q}=L^{p_\theta,q}(\R^d),\qquad (\Bs_{pq_0}^{s_0}(\R^d),\Bs_{pq_1}^{s_1}(\R^d))_{\theta,q}=\Bs_{p,q}^{s_\theta}(\R^d).
    \end{equation}
    See \cite[Theorem~4.3]{HolmstedtInterpolation} and \cite[Theorem~4.25]{SawanoBook} respectively.

    On the other hand, for $r>0$ (with $r<d$ in \eqref{Eqn::SobEmbed::Lp1}) we have
    \begin{gather}
    \label{Eqn::SobEmbed::Lp1}
        L^{\frac{dp^*}{d+rp^*}}(\R^d)\hookrightarrow\Bs_{p^*,p^*}^{-r}(\R^d),\qquad \tfrac d{d-r}<p^*\le\infty;
        \\
    \label{Eqn::SobEmbed::Lp2}
        \Bs_{\tilde p,\tilde p}^{r}(\R^d)\hookrightarrow L^\frac{d\tilde p}{d-r\tilde p}(\R^d),\qquad \tfrac d{d+r}<\tilde p<\tfrac dr.
    \end{gather}
    See e.g. \cite[Remark~2.7.1/3]{TriebelTheoryOfFunctionSpacesI} where in the reference $L^p(\R^d)=\Fs_{p2}^0(\R^d)$ by \cite[Theorem~2.5.6]{TriebelTheoryOfFunctionSpacesI}.

    Therefore, by applying \eqref{Eqn::SobEmbed::RealInterpo} and \eqref{Eqn::SobEmbed::Lp1} such that $r_0<r_\theta=m-s<r_1$, we get
    \begin{equation*}
        L^{p,q}=L^{\frac{dp^*}{d+r_\theta p^*},q}=(L^\frac{dp^*}{d+r_0p^*},L^\frac{dp^*}{d+r_1p^*})_{\theta,q}\hookrightarrow(\Bs_{p^*,p^*}^{-r_0},\Bs_{p^*,p^*}^{-r_1})_{\theta,q}=\Bs_{p^*,q}^{s-m}.
    \end{equation*}
    Using \eqref{Eqn::SobLorNorm} and \eqref{Eqn::BsNormDer} we obtain \ref{Item::SobEmbed::1}.

    Applying \eqref{Eqn::SobEmbed::RealInterpo} and \eqref{Eqn::SobEmbed::Lp2} such that $r_0<r_\theta=s<r_1$ we get \ref{Item::SobEmbed::3}:
    \begin{equation*}
        \Bs_{\tilde p,q}^{s}=(\Bs_{\tilde p,\tilde p}^{r_0},\Bs_{\tilde p,\tilde p}^{r_1})_{\theta,q}\hookrightarrow(L^\frac{d\tilde p}{d-r_0\tilde p},L^\frac{d\tilde p}{d-r_1\tilde p})_{\theta,q}=L^{\frac{d\tilde p}{d-r_\theta\tilde p},q}=L^{p,q}.
    \end{equation*}

    When $p=1$, by \cite[(1.2)]{SpectorL1}  we have $\|(-\Delta)^{-\frac r2}\nabla u\|_{L^{\frac d{d-r},1}}\lesssim\|\nabla u\|_{L^1}$ for $u\in W^{1,1}(\R^d)$ and $0<r<d$ (this is a refinement to the result \cite{SchikorraSpectorVanSchaftingenL1} which gives $\|(-\Delta)^{-\frac r2}\nabla u\|_{L^{\frac d{d-r}}}\lesssim\|\nabla u\|_{L^1}$). By the H\"ormander-Mikhlin multiplier theorem and $1<\frac d{d-r}<\infty$, $(I-\Delta)^{-\frac r2}(-\Delta)^{\frac r2}:L^{\frac d{d-r},1}\to L^{\frac d{d-r},1}$ is bounded. Thus $\|\nabla(I-\Delta)^{-\frac r2}u\|_{L^{\frac d{d-r},1}}\lesssim\|\nabla u\|_{L^1}$. The low frequency part $\varphi_0*(I-\Delta)^{-\frac r2}u$ is bounded in $L^{\frac d{d-r},1}$ by $\|u\|_{L^1}$ because its convolution kernel is Schwartz; the remaining part is controlled by $\|\nabla(I-\Delta)^{-\frac r2}u\|_{L^{\frac d{d-r},1}}$ by the H\"ormander-Mikhlin multiplier theorem. Hence $(I-\Delta)^{-\frac r2}:W^{1,1} (\mathbb{R}^{d})\to W^1L^{\frac d{d-r},1}$. Applying \ref{Item::SobEmbed::1} we have $W^{1} L^{\frac d{d-r},1}\hookrightarrow\Bs_{\frac{d}{d-2r},1}^{1-r}$ for $0<r\le\frac d2$. Taking $r=\frac{m-s}2$ and using \eqref{Eqn::BsNormDer}, we get $W^{1,1}\hookrightarrow\Bs_{\frac d{d-m+s},1}^{s-m+1}$. Using \eqref{Eqn::SobLorNorm} and \eqref{Eqn::BsNormDer}, we obtain $W^{m,1} \hookrightarrow\Bs_{\frac d{d-m+s},1}^s$ which is \ref{Item::SobEmbed::2}.
\end{proof}

\begin{rmk}\label{Rmk::Spector}
    We do not know whether or not $W^1L^{1,q}(\R^d)\subset\Bs_{\frac d{d-r},q}^{1-r}(\R^d)$ holds for $0<q<1$ and $0<r\le d$, with $\frac d{d-r}=\infty$ when $r=d$. In particular we do not know whether we have the analogue of Spector's estimate $\|(-\Delta)^{-\frac r2}\nabla u\|_{L^{\frac d{d-r},q}}\lesssim\|\nabla u\|_{L^{1,q}}$ for $u\in W^1L^{1,q}(\R^d)$, $0<q<1$ and $0<r<d$. This is the reason why we need to assume $q=1$ when $p=1$ in Theorem~\ref{Thm::MainThm}.
\end{rmk}

Now we can start to prove Theorem~\ref{Thm::MainThm} using the proof ingredients from \cite[Theorem~1.2~(ii)]{ChuahLangYaoBesov}, which we recall briefly here. Consider the embedding $\Bs_{p_0q_0}^{s_0}(\R^d)\hookrightarrow\Bs_{p_1q_1}^{s_1}(\R^d)$ where $0<q_0<q_1\le\infty$ and $s_0-s_1=\frac d{p_0}-\frac d{p_1}>0$. Using wavelet decomposition we obtain isomorphisms between Besov spaces and double sequence spaces: there exists an isomorphism\footnote{Here we use $T:X\xrightarrow{\simeq}Y$ to denote that $T$ is a bounded bijective map such that $T^{-1}:Y\to X$ is bounded as well.} (see \cite[Theorem~1.20]{TriebelSpacesOnDomains})
\begin{equation}\label{Eqn::WaveletIso}
    \Gamma:\Bs_{p_1q_1}^{s_1}(\R^d)\xrightarrow{\simeq}\ell^{q_1}(\ell^{p_1}),\quad\text{such that}\quad\Gamma:\Bs_{p_0q_0}^{s_0}(\R^d)\xrightarrow{\simeq}\ell^{q_0}(\ell^{p_0})\quad\text{ is an isomorphism as well}.
\end{equation}
 This is done by multiplying a weighted sequence to the standard wavelet decomposition maps, see e.g. \cite[Proposition~2.13]{ChuahLangYaoBesov} for details. As a result $b_n\big(\Bs_{p_0q_0}^{s_0}\hookrightarrow\Bs_{p_1q_1}^{s_1}\big)\le Cb_n\big(\ell^{q_0}(\ell^{p_0})\hookrightarrow\ell^{q_1}(\ell^{p_1})\big)$ where $C$ depends on the operator norms of $\Gamma$ and $\Gamma^{-1}$.

On the other hand, using Plichko's method \cite{PlichkoStrictlySingular}, we have an estimate given in \cite[Proposition~3.5]{ChuahLangYaoBesov}: for $0<p_0<p_1\le\infty$ and $0<q_0<q_1\le\infty$,
\begin{equation}\label{Eqn::SeqEmbedBn}
    b_n\big(\ell^{q_0}(\ell^{p_0})\hookrightarrow\ell^{q_1}(\ell^{p_1})\big)\le n^{-\min(\frac1{p_0},\frac1{q_0})(1-\max(\frac{q_0}{q_1},\frac{p_0}{p_1}))}.
\end{equation}

Let us recall the proof of \eqref{Eqn::SeqEmbedBn}. For every $n$-dimensional subspace $V\subset\ell^{q_0}(\ell^{p_0})$, we have $V\subset c_0$ (the space of all sequences which tend to $0$ at infinity). Hence (see \cite[Lemma~4]{PlichkoStrictlySingular}), there is a nonzero $x=(x_{j,k})_{j,k=0}^\infty\in V$ such that $\sup_{j,k}|x_{j,k}|$ is attained at at least $n$ components, i.e. there are at least $n$ many $(j,k)$s such that $|x_{j,k}|=\sup_{j',k'}|x_{j',k'}|$. By normalization, we assume $\|x\|_{\ell^{q_0}(\ell^{p_0})}=1$ hence $\|x\|_{\ell^{\max(p_0,q_0)}}\le1$, which implies $\|x\|_{\ell^\infty}\le n^{-\min(\frac1{p_0},\frac1{q_0})}$. Meanwhile, by H\"older's interpolation inequality $\|x\|_{\ell^{q_1}(\ell^{p_1})}\le \|x\|_{\ell^{q_0}(\ell^{p_0})}^{\max(\frac{q_0}{q_1},\frac{p_0}{p_1})}\|x\|_{\ell^\infty}^{1-\max(\frac{q_0}{q_1},\frac{p_0}{p_1})}$, we obtain $\|x\|_{\ell^{q_1}(\ell^{p_1})}\le n^{-\min(\frac1{p_0},\frac1{q_0})(1-\max(\frac{q_0}{q_1},\frac{p_0}{p_1}))}$, the desired estimate \eqref{Eqn::SeqEmbedBn}.

\begin{rmk}
    Note that an improvement of \eqref{Eqn::SeqEmbedBn} can directly lead to an improvement of \eqref{Eqn::MainBnEst}.
\end{rmk}

\begin{proof}[Proof of Theorem~\ref{Thm::MainThm}]
    Notice that $W^m_0L^{p,q}(\Omega)\subseteq W^{m}L^{p,q}(\R^d)$ is a closed subspace for an arbitrary open set $\Omega\subseteq\R^d$, with the same Sobolev norms. Therefore (see also \cite[Lemma~4.1]{ChuahLangYaoBesov}), $$b_n\big(W^m_0L^{p,q}(\Omega)\hookrightarrow L^{p^*,r}(\Omega)\big)\le b_n\big(W^m_0L^{p,q}(\R^d)\hookrightarrow L^{p^*,r}(\R^d)\big)$$ and it suffices to consider the case $\Omega=\R^d$.

    Let $\delta>0$ be a small number. By Lemma~\ref{Lem::SobEmbed}, we have a factorization:
    \begin{gather}
        W^mL^{p,q}(\R^d)\hookrightarrow\Bs_{p+\delta,q}^{m-\frac{d\delta}{p(p+\delta)}}(\R^d)\xhookrightarrow{\iota}\Bs_{p^*-\delta,r}^{\frac{d\delta}{p^*(p^*-\delta)}}(\R^d)\hookrightarrow L^{p^*,r}(\R^d),
    \end{gather}
where $\iota$ denotes the usual embedding between the Besov spaces.

    Using wavelet decomposition \eqref{Eqn::WaveletIso}, there exists an isomorphism $\Gamma:\Bs_{p^*-\delta,r}^{\frac{d\delta}{p^*(p^*-\delta)}}(\R^d)\xrightarrow{\simeq}\ell^r(\ell^{p^*-\delta})$ such that the restriction $\Gamma:\Bs_{p+\delta,q}^{m-\frac{d\delta}{p(p+\delta)}}(\R^d)\xrightarrow{\simeq}\ell^q(\ell^{p+\delta})$ is also an isomorphism.  Therefore, by compositions:
    \begin{equation*}
        b_n\big(W^mL^{p,q}(\R^d)\hookrightarrow L^{p^*,r}(\R^d)\big)\lesssim_\delta b_n\big(\Bs_{p+\delta,q}^{m-\frac{d\delta}{p(p+\delta)}}(\R^d)\hookrightarrow\Bs_{p^*-\delta,r}^{\frac{d\delta}{p^*(p^*-\delta)}}(\R^d)\big)\lesssim_\delta b_n\big(\ell^q(\ell^{p+\delta})\hookrightarrow\ell^r(\ell^{p^*-\delta})\big).
    \end{equation*}
    Note that by \eqref{Eqn::SeqEmbedBn} (see \cite[Proposition~3.5]{ChuahLangYaoBesov} for full proof), we have
    \begin{gather*}
        b_n\big(\ell^q(\ell^{p+\delta})\hookrightarrow\ell^r(\ell^{p^*-\delta})\big)\le n^{-\min(\frac1{p+\delta},\frac1q)(1-\max(\frac qr,\frac{p+\delta}{p^*-\delta}))}.
    \end{gather*}

    Now, for $\eps>0$ in the assumption, taking $0<\delta<\eps$ such that $\frac{p+\delta}{p^*-\delta}<\frac p{p^*}+\eps$ we get \eqref{Eqn::MainBnEst} and complete the proof.
\end{proof}
Notice that when $p<q<r<p^*$, for $\delta > 0$ sufficiently small, $\min(\frac1{p+\delta},\frac1q)(1-\max(\frac qr,\frac{p+\delta}{p^*-\delta}))=\frac1q-\frac1r$. This exponent is sharp by the following:
\begin{prop}\label{Prop::SobSharp}
    $m\ge1$, $1\le p<d/m$, $0<q\le r\le\infty$ such that $q\le1$ if $p=1$. Let $p^*:=\frac{dp}{d-mp}$.

    Let $\Omega\subseteq\R^d$ be an open subset. Then there is a $c>0$ such that $$b_n(W_0^mL^{p,q}(\Omega)\hookrightarrow L^{p^*,r}(\Omega))\ge c\cdot n^{1/r-1/q},\qquad n\ge1.$$

    Moreover when $q=r$ the embedding $W_0^mL^{p,q}(\Omega)\hookrightarrow L^{p^*,r}(\Omega)$ is not strictly singular.
\end{prop}

The key is to obtain a subspace of $L^{p,q}(\Omega)$ which is isomorphic to $\ell^q$.

\begin{lem}\label{Lem::LorentzSum}
    Let $0<p<\infty$ and $0<q\le\infty$. There is a $C=C_{p,q}>0$ such that for every bounded sequence $\alpha=(\alpha_j)_{j\in\Z}$,
    \begin{equation}\label{Eqn::LorentzSum::lqSubspace}
        C^{-1}\|\alpha\|_{\ell^q(\Z)}\le\Big\|\sum_{j\in\Z}\alpha_j2^\frac jp\1_{[2^{-j},2^{1-j})}\Big\|_{L^{p,q}(\R_+)}\le C\|\alpha\|_{\ell^q(\Z)}.
    \end{equation}
\end{lem}

\begin{proof}
    Fix a bounded sequence $\alpha:\Z\to\C$. We define $\beta_j:=\sup_{k\le j}2^{\frac{k-j}p}|\alpha_k|$ for $j\in\Z$. Therefore $|\alpha_j|\le \beta_j\le\sum_{k\in\Z}2^{-|j-k|/p}|\alpha_k|$ for every $j\in\Z$, which means
    \begin{equation}\label{Eqn::LorentzSum::ABSeq}
        \|\alpha\|_{\ell^q}\le\|\beta\|_{\ell^q}\le\|2^{-\frac1p|\cdot|}\ast|\alpha|\|_{\ell^q}\lesssim_{p,q}\||\alpha|\|_{\ell^q}=\|\alpha\|_{\ell^q}.
    \end{equation}
    Here we use the inequality $\|2^{-\delta|\cdot|}\ast\alpha\|_{\ell^q}\lesssim_{\delta,q}\|\alpha\|_{\ell^q}$ for all $\delta>0$ and $0<q\le\infty$. Indeed, when $q\ge1$, this is just Young's inequality; when $q<1$, we have
    \begin{equation*}
        \textstyle \|2^{-\delta|\cdot|}\ast\alpha\|_{\ell^q}^q=\sum_{j\in\Z}\big|\sum_{k\in\Z}2^{-\delta|j-k|}\alpha_k\big|^q\le \sum_{j,k\in\Z}2^{-\delta q|j-k|}|\alpha_k|^q\lesssim_{\delta q}\sum_{j\in\Z}|\alpha_j|^q.
    \end{equation*}

    Now we define $f_\alpha:=\sum_{j\in\Z}\alpha_j2^{j/p}\1_{[2^{-j},2^{1-j})}$ on $\R_+$. Note that $|f_\alpha|=f_{|\alpha|}\le f_\beta$.

    Since $(2^{j/p}\beta_j)_{j\in\Z}$ is an increasing sequence we see that $f_\beta$ is a decreasing function on $\R_+$. Therefore $f_\beta=f_\beta^*$ equals its decreasing rearrangement. We get
    \begin{align*}
        \|f_\beta\|_{L^{p,q}(\R_+)}=\Big(\int_0^\infty t^{\frac qp-1}f_\beta^*(t)^qdt\Big)^{1/q}=\Big(\sum_{j\in\Z}\beta_j^q\int_{2^{-j}}^{2^{1-j}}t^{\frac qp-1}2^{j\frac qp}dt\Big)^{1/q}\approx_{p,q}\Big(\sum_{j\in\Z}\beta_j^q\Big)^{1/q}=\|\beta\|_{\ell^q}.
    \end{align*}

    Conversely, note that for functions $u_j:\R_+\to\R_+$ we have $(\sum_ju_j)^*(t)\ge\sup_ju_j^*(t)$. Taking $u_j=2^{j/p}|\alpha_j|\1_{[2^{-j},2^{1-j})}$ and using the convention $\tau_m\beta=(\beta_{j-m})_{j\in\Z}$ for $m\in\Z$, we get
    \begin{align*}
        f_{|\alpha|}^*(t)\ge&\sup_{j\in\Z}2^{\frac jp}|\alpha_j|\1_{(0,2^{-j})}(t)=\sum_{k\in\Z}\sup_{j\le k}2^{\frac jp}|\alpha_j|\1_{[2^{-k-1},2^{-k})}(t)
        \\
        =&\sum_{k\in\Z}2^{\frac kp}\beta_k\1_{[2^{-k-1},2^{-k})}(t)=2^{-\frac1p}\sum_{k\in\Z}2^{\frac kp}\beta_{k-1}\1_{[2^{-k},2^{1-k})}(t)=2^{-\frac1p}f_{\tau_1\beta}(t).
    \end{align*}

    Therefore $\|f_\alpha\|_{L^{p,q}(\R_+)}\ge2^{-1/p}\|f_\beta\|_{L^{p,q}}$, we conclude that $\|f_\alpha\|_{L^{p,q}}\approx_{p,q}\|f_\beta\|_{L^{p,q}}\approx\|\beta\|_{\ell^q}\approx\|\alpha\|_{\ell^q}$, finishing the proof.
\end{proof}

\begin{proof}[Proof of Proposition~\ref{Prop::SobSharp}]
    Without loss of generality, we can assume $\Omega$ is bounded. Therefore, by Sobolev's inequality, $\|\nabla^mf\|_{L^{p,q}}:=\sum_{|\alpha|=m}\|\partial^\alpha f\|_{L^{p,q}}$ is an equivalent norm for the space $W_0^mL^{p,q}(\Omega)$.

    Let $\rho>0$ and $x_0\in \Omega$ be such that $B(x_0,2\rho)\subset \Omega$. Therefore there is a sequence $(x_j)_{j=1}^\infty\subset B(x_0,2\rho)$ such that $\{B(x_j,2^{-j}\rho)\}_{j=0}^\infty$ are disjointed subsets of $\Omega$.

    Fix a nonzero function $\tilde u\in C_c^\infty(B(0,\rho))$, we define $u_j(x):=2^{jd/p^*}\tilde u(2^j(x-x_j))$ for $j\ge0$. Recall the scaling property $\|f(\lambda\cdot)\|_{L^{p,q}}=\lambda^{-d/p}\|f\|_{L^{p,q}}$ for every $\lambda>0$ and $f\in L^{p,q}(\R^d)$, therefore $\|\nabla^mu_j\|_{L^{p,q}}=\|\nabla^m\tilde u\|_{L^{p,q}}$ and $\|u_j\|_{L^{p^*,r}}=\|\tilde u\|_{L^{p^*,r}}$ for all $j\ge0$.

    Note that $(u_j)_{j=0}^\infty$ have disjoint supports. Applying Lemma~\ref{Lem::LorentzSum} with a suitable change of measures, we see that for every sequence $\alpha=(\alpha_j)_{j=0}^\infty:\Z_{\ge0}\to\C$,
    \begin{gather*}
        \Big\|\sum_{j=0}^\infty\alpha_j u_j\Big\|_{W^{m}L^{p,q}}\approx \Big\|\sum_{j=0}^\infty\alpha_j \cdot\nabla^mu_j\Big\|_{L^{p,q}}=\Big\|\sum_{j=0}^\infty\alpha_j \cdot|\nabla^mu_j|\Big\|_{L^{p,q}}\approx\|\alpha\|_{\ell^q} \|\nabla^m\tilde u\|_{L^{p,q}} ;
        \\
        \Big\|\sum_{j=0}^\infty\alpha_j u_j\Big\|_{L^{p^*,r}}=\Big\|\sum_{j=0}^\infty\alpha_j \cdot|u_j|\Big\|_{L^{p^*,r}}\approx\|\alpha\|_{\ell^r} \|\tilde u\|_{L^{p^*,r}}.
    \end{gather*}

    When $q<r$, for every $n$, we can take a subspace $V_n:=\Span(u_0,\dots,u_{n-1})$. The above estimates yield $\|f\|_{W^{m}L^{p,q}}\lesssim n^{1/q-1/r}\|f\|_{L^{p^*,r}}$ for all $f\in V_n$ with the implied constant independent of $n$. Thus, we get $b_n\gtrsim n^{1/r-1/q}$.

    When $q=r$, we can take the infinite dimensional subspace $V_\infty:=\Span(u_0,u_1,\dots)$. The above estimates yield $\|f\|_{W^{m}L^{p,q}}\approx\|f\|_{L^{p^*,r}}$ for all $f\in V_\infty$, proving the non-strict singularity.
\end{proof}

\begin{ack}
    The authors thank Armin Schikorra for the discussion on $W^{1,1}$ estimates.
\end{ack}

\subsection*{AI disclaimer}
ChatGPT was used to assist with proofreading and typographical corrections for this arXiv submission. The structure and main arguments of the paper remain unchanged from the published version. The authors have reviewed all changes and take full responsibility for the mathematical content of this paper.

\bibliographystyle{amsalpha}
\bibliography{reference}

\end{document}